\documentclass{amsart}
\usepackage{amsmath}
\usepackage{amssymb}
\usepackage{amsfonts}

\newtheorem{theorem}{Theorem}
\theoremstyle{plain}

\newtheorem{definition}{Definition}

\newtheorem{lemma}{Lemma}

\numberwithin{equation}{section}
\input{tcilatex}

\begin{document}
\title[]{A NEW GENERALIZATION OF THE MIDPOINT FORMULA FOR $n-$TIME
DIFFERENTIABLE MAPPINGS WHICH ARE CONVEX}
\author{M. EM\.{I}N \"{O}ZDEM\.{I}R$^{\bigstar }$}
\address{$^{\bigstar }$ATAT\"{U}RK UNIVERSITY, K. K. EDUCATION FACULTY,
DEPARTMENT OF MATHEMATICS, 25240, CAMPUS, ERZURUM, TURKEY}
\email{emos@atauni.edu.tr}
\author{\c{C}ET\.{I}N YILDIZ$^{\bigstar ,\spadesuit }$}
\email{cetin@atauni.edu.tr}
\thanks{$^{\spadesuit }$Corresponding Author.}

\begin{abstract}
In this paper, we establish several new inequalities for $n-$ time
differentiable mappings that are connected with the celebrated
Hermite-Hadamard integral inequality.
\end{abstract}

\maketitle

\section{INRODUCTION}

On November 22, 1881, Hermite (1822-1901) sent a letter to the Journal
Mathesis. This letter was published in Mathesis 3 (1883, p: 82) and in this
letter an inequality presented which is well-known in the literature as
Hermite-Hadamard integral inequality:%
\begin{equation}
f\left( \frac{a+b}{2}\right) \leq \frac{1}{b-a}\int_{a}^{b}f(x)dx\leq \frac{%
f(a)+f(b)}{2}  \label{1}
\end{equation}%
where $f:I\subseteq 
\mathbb{R}
\rightarrow 
\mathbb{R}
$ is a convex function on the interval $I$ of a real numbers and $a,b\in I$
with $a<b.$ If the function $f$ is concave, the inequality in (\ref{1}) is
reversed.

The inequalities (\ref{1}) have become an important cornerstone in
mathematical anlysis and optimization. Many uses of these inequalities have
been discovered in a variety of settings. Moreover , many inequalities of
special means can be obtained for a particular choice of the function $f$.
Due to the rich geometrical significance of Hermite-Hadamard's inequlity,
there is growing literature providing its new proofs, extensions,
refinements and generalizations, see for example (\cite{alom}, \cite{drag}, 
\cite{jian}-\cite{kirrr}, \cite{az}-\cite{xi}) and the references therein.

\begin{definition}
A function $f:[a,b]\subset 
\mathbb{R}
\rightarrow 
\mathbb{R}
$ is said to be convex if whenever $x,y\in \lbrack a,b]$ and $t\in \lbrack
0,1]$, the following inequality holds:%
\begin{equation*}
f(tx+(1-t)y)\leq tf(x)+(1-t)f(y).
\end{equation*}
\end{definition}

We say that $f$ is concave if ($-f$) is convex. This definition has its
origins in Jensen's results from \cite{d} and has opened up the most
extended, useful and multi-disciplinary domain of mathematics, namely,
convex analysis. Convex curves and convex bodies have appeared in
mathematical literature since antiquity and there are many important results
related to them.

For other recent results concerning the $n-$time differentiable functions
see \cite{bai}-\cite{ceron}, \cite{hang}, \cite{jian}, \cite{yii}, \cite%
{wang} where further references are given.

The main purpose of the present paper is to establish several new
inequalities for $n-$ time differantiable mappings that are connected with
the celebrated Hermite-Hadamard integral inequality.

\section{MAIN RESULTS}

\begin{lemma}
\label{2}For $n\in 
\mathbb{N}
$; let $f:I\subset 
\mathbb{R}
\rightarrow 
\mathbb{R}
$ be n-time differentiable. If $a,b\in I$ with $a<b$ and $f^{(n)}\in L[a,b]$%
, then%
\begin{eqnarray}
\int_{a}^{b}f(t)dt &=&\overset{n-1}{\underset{k=0}{\sum }}\left( \frac{%
1+(-1)^{k}}{2^{k+1}(k+1)!}\right) (b-a)^{k+1}f^{(k)}\left( \frac{a+b}{2}%
\right)   \label{a} \\
&&+(b-a)^{n+1}\int_{0}^{1}M_{n}(t)f^{(n)}(ta+(1-t)b)dt  \notag
\end{eqnarray}%
where%
\begin{equation*}
M_{n}(t)=\left\{ 
\begin{array}{cc}
\frac{t^{n}}{n!}, & t\in \left[ 0,\frac{1}{2}\right]  \\ 
&  \\ 
\frac{(t-1)^{n}}{n!}, & t\in \left( \frac{1}{2},1\right] .%
\end{array}%
\right. 
\end{equation*}%
and $n$ natural number, $n\geq 1$.
\end{lemma}

\begin{proof}
The proof is by mathematical induction.

The case $n=1$ is [\cite{kir}, Lemma 2.1].

Assume that (\ref{a}) holds for "$n$" and let us prove it for "$n+1$". \
That is, we have to prove the equality%
\begin{eqnarray}
\int_{a}^{b}f(t)dt &=&\overset{n}{\underset{k=0}{\sum }}\left( \frac{%
1+(-1)^{k}}{2^{k+1}(k+1)!}\right) (b-a)^{k+1}f^{(k)}\left( \frac{a+b}{2}%
\right)  \label{b} \\
&&+(b-a)^{n+2}\int_{0}^{1}M_{n+1}(t)f^{(n+1)}(ta+(1-t)b)dt  \notag
\end{eqnarray}%
where, obviously,%
\begin{equation*}
M_{n+1}(t)=\left\{ 
\begin{array}{cc}
\frac{t^{n+1}}{(n+1)!}, & t\in \left[ 0,\frac{1}{2}\right] \\ 
&  \\ 
\frac{(t-1)^{n+1}}{(n+1)!}, & t\in \left( \frac{1}{2},1\right] .%
\end{array}%
\right.
\end{equation*}%
We have%
\begin{eqnarray*}
&&(b-a)^{n+2}\int_{0}^{1}M_{n+1}(t)f^{(n+1)}(ta+(1-t)b)dt \\
&=&(b-a)^{n+2}\left\{ \int_{0}^{\frac{1}{2}}\frac{t^{n+1}}{(n+1)!}%
f^{(n+1)}(ta+(1-t)b)dt\right. \\
&&\text{ \ \ \ \ }\left. +\int_{\frac{1}{2}}^{1}\frac{(t-1)^{n+1}}{(n+1)!}%
f^{(n+1)}(ta+(1-t)b)dt\right\}
\end{eqnarray*}%
and integrating by parts gives%
\begin{eqnarray*}
&&(b-a)^{n+2}\int_{0}^{1}M_{n+1}(t)f^{(n+1)}(ta+(1-t)b)dt \\
&=&(b-a)^{n+2}\left\{ \frac{t^{n+1}}{(n+1)!}\left. \frac{f^{(n)}(ta+(1-t)b)}{%
a-b}\right\vert _{0}^{\frac{1}{2}}-\frac{1}{a-b}\int_{0}^{\frac{1}{2}}\frac{%
t^{n}}{n!}f^{(n)}(ta+(1-t)b)dt\right. \\
&&\text{ \ \ }\left. +\frac{(t-1)^{n+1}}{(n+1)!}\left. \frac{%
f^{(n)}(ta+(1-t)b)}{a-b}\right\vert _{\frac{1}{2}}^{1}-\frac{1}{a-b}\int_{%
\frac{1}{2}}^{1}\frac{(t-1)^{n}}{n!}f^{(n)}(ta+(1-t)b)dt\right\} \\
&=&-\frac{1+(-1)^{n}}{2^{n+1}(n+1)!}f^{(n)}\left( \frac{a+b}{2}\right)
(b-a)^{n+1}+(b-a)^{n+1}\int_{0}^{1}M_{n}(t)f^{(n)}(ta+(1-t)b)dt.
\end{eqnarray*}%
That is%
\begin{eqnarray*}
(b-a)^{n+1}\int_{0}^{1}M_{n}(t)f^{(n)}(ta+(1-t)b)dt &=&\frac{1+(-1)^{n}}{%
2^{n+1}(n+1)!}f^{(n)}\left( \frac{a+b}{2}\right) (b-a)^{n+1} \\
&&+(b-a)^{n+2}\int_{0}^{1}M_{n+1}(t)f^{(n+1)}(ta+(1-t)b)dt.
\end{eqnarray*}%
Now, using the mathematical induction hypothesis, we get%
\begin{eqnarray*}
\int_{a}^{b}f(t)dt &=&\overset{n-1}{\underset{k=0}{\sum }}\left( \frac{%
1+(-1)^{k}}{2^{k+1}(k+1)!}\right) (b-a)^{k+1}f^{(k)}\left( \frac{a+b}{2}%
\right) \\
&&+\frac{1+(-1)^{n}}{2^{n+1}(n+1)!}(b-a)^{n+1}f^{(n)}\left( \frac{a+b}{2}%
\right) +(b-a)^{n+2}\int_{0}^{1}M_{n+1}(t)f^{(n+1)}(ta+(1-t)b)dt \\
&=&\overset{n}{\underset{k=0}{\sum }}\left( \frac{1+(-1)^{k}}{2^{k+1}(k+1)!}%
\right) (b-a)^{k+1}f^{(k)}\left( \frac{a+b}{2}\right) \\
&&+(b-a)^{n+2}\int_{0}^{1}M_{n+1}(t)f^{(n+1)}(ta+(1-t)b)dt.
\end{eqnarray*}%
That is, the identity (\ref{b}) and the theorem is thus proved.
\end{proof}

\begin{theorem}
\label{8}For $n\geq 1,$ let $f:I\subset 
\mathbb{R}
\rightarrow 
\mathbb{R}
$ be $n-$time differentiable and $a<b.$ If $f^{(n)}\in L[a,b]$ and $%
\left\vert f^{(n)}\right\vert $ is convex on $[a,b],$ then the following
inequality holds:%
\begin{eqnarray}
&&\left\vert \int_{a}^{b}f(t)dt-\overset{n-1}{\underset{k=0}{\sum }}\left( 
\frac{1+(-1)^{k}}{2^{k+1}(k+1)!}\right) (b-a)^{k+1}f^{(k)}\left( \frac{a+b}{2%
}\right) \right\vert   \label{3} \\
&\leq &\frac{(b-a)^{n+1}}{2^{n}(n+1)!}\left( \frac{\left\vert
f^{(n)}(a)\right\vert +\left\vert f^{(n)}(b)\right\vert }{2}\right) .  \notag
\end{eqnarray}
\end{theorem}

\begin{proof}
Since $\left\vert f^{(n)}\right\vert $ is convex on $[a,b]$, from Lemma \ref%
{2} and H\"{o}lder integral inequality, it follows that%
\begin{eqnarray*}
&&\left\vert \int_{a}^{b}f(t)dt-\overset{n-1}{\underset{k=0}{\sum }}\left( 
\frac{1+(-1)^{k}}{2^{k+1}(k+1)!}\right) (b-a)^{k+1}f^{(k)}\left( \frac{a+b}{2%
}\right) \right\vert  \\
&\leq &(b-a)^{n+1}\int_{0}^{1}\left\vert M_{n}(t)\right\vert \left\vert
f^{(n)}(ta+(1-t)b)\right\vert dt \\
&\leq &(b-a)^{n+1}\left\{ \int_{0}^{\frac{1}{2}}\frac{t^{n}}{n!}\left[
t\left\vert f^{(n)}(a)\right\vert +(1-t)\left\vert f^{(n)}(b)\right\vert %
\right] dt\right.  \\
&&\text{ \ \ \ \ \ \ \ \ \ \ \ }\left. +\int_{\frac{1}{2}}^{1}\frac{(1-t)^{n}%
}{n!}\left[ t\left\vert f^{(n)}(a)\right\vert +(1-t)\left\vert
f^{(n)}(b)\right\vert \right] dt\right\}  \\
&=&\frac{(b-a)^{n+1}}{2^{n}(n+1)!}\left( \frac{\left\vert
f^{(n)}(a)\right\vert +\left\vert f^{(n)}(b)\right\vert }{2}\right) .
\end{eqnarray*}%
This completes the proof.
\end{proof}

\begin{theorem}
Let $f:I\subset 
\mathbb{R}
\rightarrow 
\mathbb{R}
$ be $n-$time differentiable and $a<b.$ If $f^{(n)}\in L[a,b]$ and $%
\left\vert f^{(n)}\right\vert ^{q}$ is convex on $[a,b],$ then we have%
\begin{eqnarray}
&&\left\vert \int_{a}^{b}f(t)dt-\overset{n-1}{\underset{k=0}{\sum }}\left( 
\frac{1+(-1)^{k}}{2^{k+1}(k+1)!}\right) (b-a)^{k+1}f^{(k)}\left( \frac{a+b}{2%
}\right) \right\vert   \label{4} \\
&\leq &\frac{(b-a)^{n+1}}{2^{n+1}n!}\left( \frac{1}{np+1}\right) ^{\frac{1}{p%
}}  \notag \\
&&\times \left\{ \left( \frac{\left\vert f^{(n)}(a)\right\vert
^{q}+3\left\vert f^{(n)}(b)\right\vert ^{q}}{4}\right) ^{\frac{1}{q}}+\left( 
\frac{3\left\vert f^{(n)}(a)\right\vert ^{q}+\left\vert
f^{(n)}(b)\right\vert ^{q}}{4}\right) ^{\frac{1}{q}}\right\}   \notag
\end{eqnarray}%
where $\frac{1}{p}+\frac{1}{q}=1.$
\end{theorem}

\begin{proof}
From Lemma \ref{2} and H\"{o}lder integral inequality, we obtain%
\begin{eqnarray*}
&&\left\vert \int_{a}^{b}f(t)dt-\overset{n-1}{\underset{k=0}{\sum }}\left( 
\frac{1+(-1)^{k}}{2^{k+1}(k+1)!}\right) (b-a)^{k+1}f^{(k)}\left( \frac{a+b}{2%
}\right) \right\vert  \\
&\leq &(b-a)^{n+1}\int_{0}^{1}\left\vert M_{n}(t)\right\vert \left\vert
f^{(n)}(ta+(1-t)b)\right\vert dt \\
&\leq &\frac{(b-a)^{n+1}}{n!}\left\{ \left( \int_{0}^{\frac{1}{2}%
}t^{np}dt\right) ^{\frac{1}{p}}\left( \int_{0}^{\frac{1}{2}}\left\vert
f^{(n)}(ta+(1-t)b)\right\vert ^{q}dt\right) ^{\frac{1}{q}}\right.  \\
&&\text{ \ \ \ }\left. +\left( \int_{\frac{1}{2}}^{1}(1-t)^{np}dt\right) ^{%
\frac{1}{p}}\left( \int_{\frac{1}{2}}^{1}\left\vert
f^{(n)}(ta+(1-t)b)\right\vert ^{q}dt\right) ^{\frac{1}{q}}\right\} .
\end{eqnarray*}%
Since $\left\vert f^{(n)}\right\vert ^{q}$ is convex on $[a,b],$ then 
\begin{eqnarray*}
&&\left\vert \int_{a}^{b}f(t)dt-\overset{n-1}{\underset{k=0}{\sum }}\left( 
\frac{1+(-1)^{k}}{2^{k+1}(k+1)!}\right) (b-a)^{k+1}f^{(k)}\left( \frac{a+b}{2%
}\right) \right\vert  \\
&\leq &\frac{(b-a)^{n+1}}{2^{n+1}n!}\left( \frac{1}{np+1}\right) ^{\frac{1}{p%
}} \\
&&\times \left\{ \left( \frac{\left\vert f^{(n)}(a)\right\vert
^{q}+3\left\vert f^{(n)}(b)\right\vert ^{q}}{4}\right) ^{\frac{1}{q}}+\left( 
\frac{3\left\vert f^{(n)}(a)\right\vert ^{q}+\left\vert
f^{(n)}(b)\right\vert ^{q}}{4}\right) ^{\frac{1}{q}}\right\} 
\end{eqnarray*}%
which completes the proof.
\end{proof}

\begin{theorem}
\label{11}For $n\geq 1,$ let $f:I\subset 
\mathbb{R}
\rightarrow 
\mathbb{R}
$ be $n-$time differentiable and $a<b.$ If $f^{(n)}\in L[a,b]$ and $%
\left\vert f^{(n)}\right\vert ^{q}$ is convex on $[a,b],$ for $q\geq 1,$
then the following inequality holds:%
\begin{eqnarray}
&&\left\vert \int_{a}^{b}f(t)dt-\overset{n-1}{\underset{k=0}{\sum }}\left( 
\frac{1+(-1)^{k}}{2^{k+1}(k+1)!}\right) (b-a)^{k+1}f^{(k)}\left( \frac{a+b}{2%
}\right) \right\vert   \label{7} \\
&\leq &\frac{(b-a)^{n+1}}{2^{n+1}(n+1)!}\left\{ \left[ \frac{n+1}{2n+4}%
\left\vert f^{(n)}(a)\right\vert ^{q}+\frac{n+3}{2n+4}\left\vert
f^{(n)}(b)\right\vert ^{q}\right] ^{\frac{1}{q}}\right.   \notag \\
&&\text{ \ \ \ \ \ \ \ \ \ \ \ \ \ }\left. +\left[ \left[ \frac{n+3}{2n+4}%
\left\vert f^{(n)}(a)\right\vert ^{q}+\frac{n+1}{2n+4}\left\vert
f^{(n)}(b)\right\vert ^{q}\right] ^{\frac{1}{q}}\right] \right\} .  \notag
\end{eqnarray}
\end{theorem}

\begin{proof}
From Lemma \ref{2} and using the well known power-mean integral inequality,
we have%
\begin{eqnarray*}
&&\left\vert \int_{a}^{b}f(t)dt-\overset{n-1}{\underset{k=0}{\sum }}\left( 
\frac{1+(-1)^{k}}{2^{k+1}(k+1)!}\right) (b-a)^{k+1}f^{(k)}\left( \frac{a+b}{2%
}\right) \right\vert  \\
&\leq &(b-a)^{n+1}\int_{0}^{1}\left\vert M_{n}(t)\right\vert \left\vert
f^{(n)}(ta+(1-t)b)\right\vert dt \\
&\leq &\frac{(b-a)^{n+1}}{n!}\left\{ \left( \int_{0}^{\frac{1}{2}%
}t^{n}dt\right) ^{1-\frac{1}{q}}\left( \int_{0}^{\frac{1}{2}}t^{n}\left\vert
f^{(n)}(ta+(1-t)b)\right\vert ^{q}dt\right) ^{\frac{1}{q}}\right.  \\
&&\text{ \ \ \ \ \ \ }\left. +\left( \int_{\frac{1}{2}}^{1}(1-t)^{n}dt%
\right) ^{1-\frac{1}{q}}\left( \int_{\frac{1}{2}}^{1}(1-t)^{n}\left\vert
f^{(n)}(ta+(1-t)b)\right\vert ^{q}dt\right) ^{\frac{1}{q}}\right\} .
\end{eqnarray*}%
Since $\left\vert f^{(n)}\right\vert ^{q}$ is convex on $[a,b],$ for $q\geq
1,$ then we obtain 
\begin{eqnarray*}
&&\left\vert \int_{a}^{b}f(t)dt-\overset{n-1}{\underset{k=0}{\sum }}\left( 
\frac{1+(-1)^{k}}{2^{k+1}(k+1)!}\right) (b-a)^{k+1}f^{(k)}\left( \frac{a+b}{2%
}\right) \right\vert  \\
&\leq &\frac{(b-a)^{n+1}}{2^{n+1}(n+1)!}\left\{ \left[ \frac{n+1}{2n+4}%
\left\vert f^{(n)}(a)\right\vert ^{q}+\frac{n+3}{2n+4}\left\vert
f^{(n)}(b)\right\vert ^{q}\right] ^{\frac{1}{q}}\right.  \\
&&\text{ \ \ \ \ \ \ \ \ \ \ \ \ }\left. +\left[ \left[ \frac{n+3}{2n+4}%
\left\vert f^{(n)}(a)\right\vert ^{q}+\frac{n+1}{2n+4}\left\vert
f^{(n)}(b)\right\vert ^{q}\right] ^{\frac{1}{q}}\right] \right\} 
\end{eqnarray*}%
whisch completes the proof.
\end{proof}


\begin{thebibliography}{99}
\bibitem{alom} M. Alomari, M. Darus and S.S. Dragomir, New inequalities of
Hermite-Hadamard type for functions whose second derivatives absolute values
are quasi-convex. Tamkang Journal of Mathematics, 41(4), 2010, 353-359.

\bibitem{bai} S.-P. Bai, S.-H. Wang and F. Qi, Some Hermite-Hadamard type
inequalities for $n$-time differentiable $(\alpha ,m)$-convex functions,
Jour. of Ineq. and Appl., 2012, 2012:267.

\bibitem{cer} P. Cerone, S.S. Dragomir and J. Roumeliotis, Some Ostrowski
type inequalities for $n$-time differentiable mappings and applications,
Demonstratio Math., 32 (4) (1999), 697-712.

\bibitem{ceron} P. Cerone, S.S. Dragomir and J. Roumeliotis and J. \v{S}%
unde, A new generalization of the trapezoid formula for $n$-time
differentiable mappings and applications, Demonstratio Math., 33 (4) (2000),
719-736.

\bibitem{drag} S.S. Dragomir and C.E.M. Pearce, Selected Topics on
Hermite-Hadamard Inequalities and Applications, RGMIA Monographs, Victoria
University, 2000.
Online:[http://www.staxo.vu.edu.au/RGMIA/monographs/hermite hadamard.html].

\bibitem{hang} D.-Y. Hwang, Some Inequalities for $n$-time Differentiable
Mappings and Applications, Kyung. Math. Jour., 43 (2003), 335-343.

\bibitem{d} J. L. W. V. Jensen, On konvexe funktioner og uligheder mellem
middlvaerdier, Nyt. Tidsskr. Math. B., 16, 49-69, 1905.

\bibitem{jian} W.-D. Jiang, D.-W. Niu, Y. Hua and F. Qi, Generalizations of
Hermite-Hadamard inequality to $n$-time differentiable function which are $s$%
-convex in the second sense, Analysis (Munich), 32 (2012), 209-220

\bibitem{kir} U.S. K\i rmac\i , Inequalities for differentiable mappings and
applications to special means of real numbers and to midpoint formula, Appl.
Math. and Comp., 147 (2004), 137--146.

\bibitem{us} U.S. K\i rmac\i , M.K. Bakula, M.E. \"{O}zdemir and J. Pe\'{c}%
ari\'{c}, Hadamard-type inequalities for $s$-convex functions, Appl. Math.
and Comp., 193(2007), 26-35.

\bibitem{kirrr} M.E. \"{O}zdemir and U.S. K\i rmac\i , Two new theorem on
mappings uniformly continuous and convex with applications to quadrature
rules and means, Appl. Math. and Comp., 143(2003), 269-274.

\bibitem{yii} M.E. \"{O}zdemir, \c{C}. Y\i ld\i z, New Inequalities for $n$%
-time differentiable functions, Arxiv:1402.4959v1.

\bibitem{yil} M.E. \"{O}zdemir, \c{C}. Y\i ld\i z, New Inequalities for
Hermite-Hadamard and Simpson Type with Applications, Tamkang J. of Math.,
44, 2, 209-216, 2013..

\bibitem{18} J.E. Pe\v{c}ari\'{c}, F. Porschan and Y.L. Tong, Convex
Functions, Partial Orderings, and Statistical Applications, Academic Press
Inc., 1992.

\bibitem{az} A. Saglam, M.Z Sar\i kaya and H. Y\i ld\i r\i m, Some new
inequalities of Hermite-Hadamard's type, Kyung. Math. Jour., 50 (2010),
399-410.

\bibitem{set} E. Set, M.E. \"{O}zdemir and S.S. Dragomir, On Hadamard-Type
Inequalities Involving Several Kinds of Convexity, Jour. of Ineq. and Appl.,
2010, 286845.

\bibitem{wang} S.H. Wang, B.-Y. Xi and F. Qi, Some new inequalities of
Hermite-Hadamard type for $n$-time differentiable functions which are $m$%
-convex, Analysis (Munich), 32 (2012), 247-262.

\bibitem{xi} B.-Y. Xi and F. Qi, Some integral inequalities of
Hermite-Hadamard type for convex functions with applications to means, J.
Funct. Spaces Appl., 2012 (2012), http://dx.doi.org/10.1155/2012/980438.
\end{thebibliography}
\end{document}